\documentclass[12pt,french]{article}
\usepackage{graphicx}
\usepackage{a4} % pour la taille
\usepackage[T1]{fontenc} % pour les font postscript
\usepackage[cyr]{aeguill} % Police vectorielle TrueType, guillemets fran¸cais
\usepackage{epsfig} % pour g´erer les images
\usepackage{amsmath, amsthm} % très bon mode mathématique
\usepackage{amsfonts,amssymb}% permet la definition des ensembles
\usepackage{float} % pour le placement des figure
\usepackage{url} % pour une gestion efficace des url
\usepackage{colortbl}% pour les couleurs des tableaus
\usepackage{graphics}
\usepackage[latin1]{inputenc}
\usepackage{tikz,tkz-tab}
\usepackage{fancyhdr}
\newtheorem{remark}{Remark}[section]
\newtheorem{definition}{Definition}[section]
\newtheorem{proposition}{Proposition}[section]
\newtheorem{lemma}{Lemma}[section]

\newtheorem{theorem}{Theorem}[section]

\date{}

\begin{document}
\title{Existence of free boundaries for some overdetermined-value problems}
\author{Mohammed Barkatou
}
\date{}
\maketitle{
\begin{abstract}
The aim of this paper is first to give necessary and sufficient condition of existence (of free boundaries) for both Laplacian and bi-Laplacian operators in the case where the overdetermined condition is not constant. Second, by using some classical inequalities, we get existence result for several (other) overdetermined free boundary problems.
\end{abstract}

\maketitle {\small {\noindent {\bf Keywords:} Functional inequalities, free boundaries, Laplacian, Bi-Laplacian, maximum principle, mean curvature, shape derivative, shape optimization, quadrature surfaces problem.

\noindent {\bf 2010 Mathematics Subject Classification:} 35A15, 35J65, 35B50

\section{Introduction}
Assuming throughout that: $D\subset \mathbb{R}^{N}\quad (N=2,\;3)$
is a hold-all which contains all the domains we use. If
$\Omega$ is an open subset of $D$, let $\nu$ be the outward normal
to $\partial\Omega$ and let $\left| \partial \Omega \right|$
(respectively $\left| \Omega \right|$) be the perimeter
(respectively the volume) of $\Omega$. Let $g$ and $f$ be positive functions  in $L^2(\mathbb{R}^{N})$ such that $f$ has a compact support
with nonempty interior. Denote by $C_f$ the convex hull of $\text{supp}f$.\\
We are interested by proving the existence of a solution for the two following overdetermined free boundary value problems.
\begin{enumerate}
  \item $\mathcal{QS}(f,g)$: Find $\Omega\subset D$ which strictly contains $C_f$ such that
  \begin{equation*}
\left\{
\begin{array}{c}
-\Delta u_{\Omega}=f \quad \text{in  }\Omega\,, \\
u_{\Omega}=0\text{ on  }\partial \Omega\,, \\
 -\frac{\partial u_{\Omega}}{\partial \nu }=g\text{  on  }\partial
 \Omega.
\end{array}
\right.
\end{equation*}
  \item $\mathcal{B}(f,g)$: Find $\Omega\subset D$ which strictly contains $C$ such that
  \begin{equation*}
\left\{
\begin{array}{c}
\Delta^2 v_{\Omega}=f \quad \text{ in  }\Omega\,,\\
 v_{\Omega}=\Delta v_{\Omega}=0\text{  on  }\partial\Omega\,, \\
-\frac{\partial(\Delta v_{\Omega})}{\partial \nu }\frac{\partial v_{\Omega}}{\partial \nu }=g\text{  on
}\partial\Omega.

\end{array}
\right.
\end{equation*}
\end{enumerate}

%%%%%%%%%%%%%%%%%%%%%%%%%%%%%%%%%%%%%%%%%%%%%%%%%%%%%%%%%%%%%%%%%%%%%%%%%%%%%%%%%%%%%%%%%%%%%%%%%%%%%%%%%%%%%%%%%%%%%%%%%%%%%
Notice that since $u_{\Omega}$ (resp. $v_{\Omega}$) vanishes on $\partial\Omega$ then
$-\frac{\partial u_{\Omega}}{\partial \nu }=|\nabla u_{\Omega}|$
($-\frac{\partial v_{\Omega}}{\partial \nu }=|\nabla v_{\Omega}|$).\\
The problem $\mathcal{QS}(f,g)$ is called the quadrature surfaces free
boundary problem and arises in many areas of physics (free
streamlines, jets, Hele-show flows, electromagnetic shaping,
gravitational problems etc.) It has been intensively studied from
different points of view, by several authors. For more details
about the methods used for solving this problem see the
\cite[Introduction]{gs}. Imposing boundary conditions for both $u_{\Omega }$ and $|\nabla
u_{\Omega }|$ on $\partial\Omega$ makes problem $\mathcal{QS}(f,g)$
\textit{overdetermined}, so that in general without any
assumptions on data this problem has no solution. When $g=k=constant$, Gustafsson and Shahgholian \cite{gs}
conclude their paper by giving \cite[Theorem 4.7]{gs} the
following sufficient condition: If $\text{Supp}f\subset B_R$ and
if $\int_{B_R}f(x)dx>(\frac{6^NNc}{3R}|B_R|)$ with
$B_{3R}\subset\Omega$ ($B_R$ being some ball of radius $R$) then
$\mathcal{QS}(f,g)$ has a solution. The method used by Gustafsson and
Shahgholian goes back to K. Friedrichs \cite{fr}, or even to T. Carleman \cite{ca}, and was
considerably developed by H. W. Alt and L. A. Caffarelli
\cite{ac}. In \cite{he}, in order to get solutions to Problem $\mathcal{QS}(f,g)$, A. Henrot used the method of subsolutions and supersolutions
introduced by A. Beurling \cite{be}. By combining the maximum principle to the
compatibility condition of the Neumann problem, Barkatou et al.
\cite{ba2} gave, $|\nabla u_{C_f}|>k\;\text{on}\;\partial C_f$ as a sufficient condition of existence for
$\mathcal{QS}(f,g)$. Later, Barkatou \cite{ba1} showed that this problem admits a solution if and only if
the condition $\int_{C_f} f(x)dx>k|\partial C_f|$ is valid.\\

The problem $\mathcal{B}(f,g)$ is equivalent to the following system $\mathcal{S}(f,g)$
\begin{equation*}
\left\{
\begin{array}{c}
-\Delta u_{\Omega}=f \quad \text{ in  }\Omega\,,
 u_{\Omega}=0\text{  on  }\partial\Omega \,, \\
-\Delta v_{\Omega}=u_{\Omega} \quad \text{ in  }\Omega\,,
v_{\Omega}=0 \text{ on  }\partial \Omega \,,
\\
\frac{\partial u_{\Omega}}{\partial \nu }\frac{\partial v_{\Omega}}{\partial \nu }=g\text{  on
}\partial\Omega.

\end{array}
\right.
\end{equation*}
The system $\mathcal{S}(1,c)$ arises from variational problem in
Probability \cite{fm}. Fromm and McDonald \cite{fm} related
this problem to the fundamental result of Serrin. Then using the
moving plane method combining with Serrin's boundary point Lemma,
they showed that if this problem admits a solution $\Omega$ then
it must be an $N$-ball. Huang and Miller \cite{hm} established the
variational formulas for maximizing the functionals (they
considered) over $C^{k}$ domains with a volume constraint and
obtained the same symmetry result for their maximizers. The existence for $\mathcal{S}(f,g)$ was first studied in \cite{kb} in the case where $g=c=constant$. By using the maximum principle, the authors showed that
if $|\nabla u_{C_f}||\nabla v_{C_f}|>c\text{  on }\partial C_f$, then
this problem has a solution. In \cite{ba3}, the author get the same result if $c|\partial C_f|^{2}<\int_{C_f}f\int_{C_f}u_{C_f}$. The aim here is to give a necessary and sufficient condition of
existence for the problems $\mathcal{QS}(f,g)$ and $\mathcal{B}(f,g)$ when $g$ is not constant. After, we will use sufficient conditions to prove existence for several overdetermined value problems. The use of the Cauchy-Schwarz's inequality is crucial for demonstrations. In some cases, we obtain an integral inequality that will either provide us with a solution of our free boundary problem or that $C_f$ is an $N$-ball.\\
The paper is organized as follows. In Section 2, we introduce some definitions and, for reader's convenience, recall some results which will use thereafter. Section 3 is dedicated to the main results (Theorems 3.1 and 3.2), their proofs and some corollaries. For Section 4, in order to get existence for several other overdetermined-value problems, some classical functional inequalities like Cauchy-Schwarz's inequality or eigenvalue inequalities are used. In Section 5, final remarks close the paper.

%%%%%%%%%%%%%%%%%%%%%%%%%%%%%%%%%%%%%%%%%%%%%%%%%%%%%%%%%%%%%%%%%%%%%%%%%%%%%%%%%%%%%%%%%%%%%%%%%%%%%%%%%%%%%%%%%%
%%%%%%%%%%%%%%%%%%%%%%%%%%%%%%%%%%%%%%%%%%%%%%%%%%%%%%%%%%%%%%%%%%%%%%%%%%%%%%%%%%%%%%%%%%%%%%%%%%%%%%%%%%%%%%%%%%%%%
\section{Preliminaries}
This section is dedicated to some definitions and results which will be useful in the sequel.
\begin{definition} \cite{ba1}
Let $C$ be a compact convex set, the bounded domain
$\Omega$ satisfies $C$-\textsc{gnp} if
\begin{enumerate}
\item  $\Omega \supset \text{int}(C)$,

\item  $\partial \Omega \setminus C$ is locally Lipschitz,

\item  for any $c\in \partial C$ there is an outward normal ray
$\Delta _{c}$ such that $\Delta _{c}\cap \Omega $ is connected,
and

\item  for every $x\in \partial \Omega \setminus C$ the inward
normal ray to $\Omega $ (if exists) meets $C$.
\end{enumerate}
\end{definition}

\begin{remark}
If $\Omega $ satisfies the $C$-\textsc{gnp} and $C$ has a nonempty
interior, then $\Omega $ is connected.
\end{remark}
Put
\begin{equation*}
\mathcal{O}_{C}=\left\{ \omega \subset D: \omega\quad
\text{satisfies }C-\textsc{gnp}\right\}.
\end{equation*}
\begin{theorem}
If $\Omega _{n}\in \mathcal{O}_{C}$\,, then there exist an open
subset $\Omega \subset D$ and a subsequence (again denoted by
$\Omega _{n}$) such that $\Omega_{n}$ converges to $\Omega$ in the Hausdorff, compact and characteristic sense and
$\Omega\in \mathcal{O}_{C}$. Furthermore, the three convergence are equivalent in $\mathcal{O}_{C}$.
\end{theorem}

Barkatou proved this theorem (see e.g. \cite{ba1}).

\begin{proposition}
Let $\left\{ \Omega _{n},\Omega \right\}\subset \mathcal{O}_{C}$
such that $\Omega _{n}\overset{H}{\longrightarrow }\Omega $\,. Let
$u_{n}$ and $u_{\Omega }$ be respectively the solutions of
$P(\Omega _{n},f)$ and $P(\Omega,f)$\,. Then $u_{n}$ converges
strongly in $H_{0}^{1}(D)$ to $u_{\Omega }$ ($u_{n}$ and
$u_{\Omega }$ are extended by zero in $D$).
\end{proposition}

This proposition was proven for $N=2$ or $3$ \cite[Theorem
4.3]{ba1}).
The goal of the following theorem is to show that the minimizers of the functionals considered here are of class $C^{2}$.
\begin{theorem}(\cite{bsl})
 Let $L$ be a compact set of $\mathbb{R}^N)$. Let $f_{n}$ be a sequence of functions defined on $L$.
 $$|\frac{\partial f_{n}}{\partial x_{i}}|\leq M,\;|\frac{\partial^{2} f_{n}}{\partial x_{i}\partial x_{j}}|\leq M,\;|\frac{\partial^{3} f_{n}}{\partial x_{i}\partial x_{j}\partial x_{k}}|\leq M.$$
We define $\Omega_{n}=\{x\in L\;:\; f_{n}(x)>0\}$ and we suppose the existence of $\alpha>0$ such that $|f_{n}(x)|+|\nabla f_{n}(x)|\geq \alpha$ for all $x\in L$. If $\Omega_{n}$ has $C$-\textsc{gnp}, then there exists $\Omega$ of class $C^{2}$ and a subsequence (still denoted by $\Omega_{n}$) which converges to $\Omega$ in the compact sense.
\end{theorem}
%%%%%%%%%%%%%%%%%%%%%%%%%%%%%%%%%%%%%%%%%%%%%%%%%%%%%%%%%%%%%%%%%%%%%%%%%%%%%%%%%%%%%%%%%%%%%%%%%%%%%%%%%%%%%%%%%%%%%%%%%%%%%%%%%%%%%%
\begin{definition}
Let $C$ be a convex set. We say that an open subset $\Omega $ has
the $C$-\textsc{sp}, if

\begin{enumerate}
\item  $\Omega \supset \text{int}(C)$\,,

\item  $\partial \Omega \setminus C$ is locally Lipschitz,

\item  for any $c\in \partial C$ there is an outward normal ray
$\Delta _{c}$ such that $\Delta _{c}\cap \Omega $ is connected,
and

\item  for all $x\in \partial \Omega \setminus C \quad K_{x}\cap
\Omega =\emptyset$\,, where $K_{x}$ is the closed cone defined by
\begin{equation*}
\left\{ y\in \mathbb{R}^{N} : (y-x).(z-x)\leq 0\,, \text{ for all
} z\in C\right\}.
\end{equation*}
\end{enumerate}
\end{definition}

\begin{remark}
$K_{x}$ is the normal cone to the convex hull of $C$ and $\{x\}$.
\end{remark}

\begin{proposition}\cite[Proposition 2.3]{ba1}
$\Omega$ has the $C$-\textsc{gnp} if and only if $\Omega $
satisfies the $C$-\textsc{sp}.
\end{proposition}

\begin{proposition}\cite[Theorem 3.5]{bt}
Let $v_{n}$ and $\;v_{\Omega }$ be respectively the solutions of
the Dirichlet problems $P(\Omega _{n},g_{n})$ and $P(\Omega ,g )$.
If $g_{n}$ converges strongly in $H^{-1}(D)$ to $g$ then $v_{n}$
converges strongly in $H_{0}^{1}(D)$ to $v_{\Omega }$ ($v_{n}$ and
$v_{\Omega }$ are extended by zero in $D$).
\end{proposition}

\begin{lemma}\cite{bz,pi}
Let $\Omega_{n}$ be a sequence of open and bounded subsets of $D$.
There exist a subsequence (again denoted by $\Omega_{n}$) and some
open subset $\Omega$ of $D$ such that
\begin{enumerate}
    \item $\Omega_{n}$ converges to $\Omega$ in the Hausdorff
    sense, and
    \item
    $|\Omega|\leq\liminf_{n\rightarrow\infty}|\Omega_{n}|$.
\end{enumerate}
\end{lemma}

\begin{lemma}
Let $u_{\Omega}\in C^{2,\alpha}(\bar\Omega)$ be the solution of $P(\Omega,N)$. Then $\Omega$ is an $N$-ball if and only if
  $|\nabla u_{\Omega}(x)|=\frac{1}{H_{\Omega}(x)},$ for every $x\in\partial\Omega$. $H_{\Omega}(x)$ is the mean curvature of $\partial\Omega$.
\end{lemma}
For the proof of this lemma, see Theorem 2.4 \cite{mp}.
\begin{theorem}
Let $u_{\Omega}$ (respectively $v_{\Omega}$) be the solution of $P(\Omega,1)$ (respectively $P(\Omega,u_{\Omega})$). If one of the following conditions is satisfied, then $\Omega$ is an $N$-ball.
\begin{enumerate}
  \item $|\nabla v_{\Omega}|=c\text{ on }\partial\Omega.$
  \item $|\nabla v_{\Omega}|=cx.\nu\text{ on }\partial\Omega.$
  \item $|\nabla v_{\Omega}|=c|\nabla u_{\Omega}|\text{ on }\partial\Omega.$
\end{enumerate}
\end{theorem}
For the proof of this theorem, see \cite{ps}.\\
As we use the standard tool of the domain derivative \cite{sz} to prove many of the propositions we state here, we recall its definition.\\

Suppose that the open $\Omega$ is of class $C^{2}$. Consider a deformation field $V\in C^{2}(\mathbb{R}^{N};\mathbb{R}^{N})$ and set $\Omega_{t}=\Omega+tV(\omega)$, $t>0$. The application $Id+tV$ (a perturbation of the identity) is a Lipschitz diffeomorphism for $t$ sufficiently small and, by definition, the derivative of $J$ at $\Omega$ in the direction $V$ is
$$dJ(\Omega,V)=\lim_{t\rightarrow 0}\frac{J(\Omega_{t})-J(\Omega)}{t}.$$
As the functional $J$ depends on the domain $\Omega$ through the solution of some Dirichlet problem, we need to define the domain derivative $u_{\Omega}^{'}$ of $u_{\Omega}$:
$$u_{\Omega}^{'}=\lim_{t\rightarrow 0}\frac{u_{\Omega_{t}}-u_{\Omega}}{t}.$$
Furthermore, $u_{\Omega}^{'}$ is the solution of the following problem:
\begin{equation}
 \left\{
\begin{array}{c}
-\Delta u'_{\Omega}=0 \quad \text{in  }\Omega\\
u_{\Omega}^{'}=-\frac{\partial u_{\Omega}}{\partial\nu}V.\nu\text{ on  }\partial \Omega.
\end{array}
\right.
\end{equation}
The domain derivative $v_{\Omega}^{'}$ of $v_{\Omega}$ (solution of $P(\omega,u_{\Omega})$) is the solution of:
\begin{equation}
 \left\{
\begin{array}{c}
-\Delta v'_{\Omega}=u_{\Omega}^{'}\quad \text{in  }\Omega\\
u_{\Omega}^{'}=-\frac{\partial u_{\Omega}}{\partial\nu}V.\nu\text{ on  }\partial \Omega.
\end{array}
\right.
\end{equation}
Now, to compute the derivative of the functionals we consider below, recall the following:
\begin{enumerate}
  \item The domain derivative of the volume is $$\int_{\partial\Omega}V.\nu d\sigma.$$
  \item The domain derivative of the perimeter is $$\int_{\partial\Omega}(N-1)H_{\Omega}V.\nu d\sigma,$$
  $H_{\Omega}$ being the mean curvature of $\partial\Omega$.
  \item Suppose that $u_{\Omega}\in H^{1}_{0}(D)$ and $\Omega$ is of class $C^{2}$, then
  \begin{itemize}
    \item[(a)] If $F(\Omega)=\int_{\Omega}u_{\Omega}^{2}dx$, then $$dF(\Omega,V)=2\int_{\Omega}u_{\Omega}u_{\Omega}^{'}dx.$$
    But $v_{\Omega}\in H^{1}_{0}(D)$ and $-\Delta v_{\Omega}=u_{\Omega}$ in $\Omega$, so by Green's formula we obtain
    $$dF(\Omega,V)=2\int_{\partial\Omega}|\nabla u_{\Omega}||\nabla v_{\Omega}|V.\nu d\sigma.$$
    \item[(b)] If $G(\Omega)=\int_{\Omega}|\nabla u_{\Omega}|^{2}dx$, then by Hadamard's formula $$dG(\Omega,V)=\int_{\partial\Omega}|\nabla u_{\Omega}|^{2}V.\nu d\sigma.$$
  \end{itemize}
\end{enumerate}
Since the set $\Omega$ satisfies the $C$-\textsc{gnp}, we ask the
deformation set $\Omega_{t}$ to satisfy the same property (for $t$
sufficiently small). The aim in the sequel is to prove that the $C$-\textsc{gnp} is stable by small deformation.\\
 $\Omega$ having the
$C$-\textsc{gnp}, by Proposition 2.2, it satisfies the
$C$-\textsc{sp}. Then
\begin{equation*}
\text{for all }x\in \partial \Omega \setminus C:K_{x}\cap \Omega
=\emptyset.
\end{equation*}
For $t$ sufficiently small, let $\Omega _{t}=\Omega +tV\left(
\Omega \right)$ be the deformation of $\Omega $ in the direction
$V$. Let $x_{t}\in \partial \Omega _{t}$. There exists $x\in
\partial \Omega $ such that $x_{t}=x+tV(x)$. Using the definition
of $K_{x_{t}}$ and the equality above, we get (for $t$ small
enough and for every displacement $V$):
\begin{equation*}
\text{for all }x_{t}\in \partial \Omega _{t}\setminus
C:K_{x_{t}}\cap \Omega _{t}=\emptyset,
\end{equation*}
which means that $\Omega _{t}$\ satisfies the $C$-\textsc{sp} (and
so the $C$-\textsc{gnp}) for every direction $V$\ when $t$ is
sufficiently small. In fact, suppose, by contradiction, there exists $x_{t}\in\partial\Omega_{t}\setminus C$ such that $K_{x_{t}}\cap\Omega_{t}\neq\emptyset$. Let $y_{t}\in K_{x_{t}}\cap\Omega_{t}$, there exists $y\in\Omega$, $y=y_{t}-tV(y)$ such that:
$$\forall c\in C,\;\;(y_{t}-x_{t}).(c-x_{t})\leq 0.$$
Show that $y\in K_{x}$:
\begin{eqnarray*}
    (y-x).(c-x)&=&(y_{t}-tV(y)-x_{t}+tV(x)).(c-x_{t}+tV(x))\\
    &=&(y_{t}-x_{t}+t(V(y)-V(x))).(c-x_{t}+tV(x))\\
    &=&(y_{t}-x_{t}).(c-x_{t})+\epsilon(t)
\end{eqnarray*}
where $\epsilon(t)=t(y_{t}-x_{t}).V(x)+t(V(y)-V(x)).(c-x_{t})+t^2(V(y)-V(x)).V(x)$ which, as $t$, tends to $0$. Obtaining the contradiction.

%%%%%%%%%%%%%%%%%%%%%%%%%%%%%%%%%%%%%%%%%%%%%%%%%%%%%%%%%%%%%%%%%%%%%%%%%%%%%%%%%%%%%%%%%%%%%%%%%%%%%%%%%%%%%
\section{Main results}
\subsection{Existence for Problem $\mathcal{QS}(f,g)$}
The purpose of this section is to establish a necessary and sufficient condition of first Problem $\mathcal{QS}(f,g)$ and then Problem $\mathcal{B}(f,g)$. Here, we are interested by the case where $int(C_f)$ is not empty (otherwise, we can replace $C_f$ by the smallest closed ball which contains $C_f$).
Before stating and proving the main theorems of this section, let us introduce (or recall) some notations.
\begin{itemize}
  \item $C_f=\text{conv}(\text{supp}f)$.
  \item $\mathcal{O}_{C}=\left\{ \Omega \subset D: \Omega\quad \text{satisfies }C-\textsc{gnp}\right\}$.
  \item $\mathcal{O}_{\varepsilon}=\{\Omega\subset D;\;\Omega\; \text{satisfies a uniform interior and exterior }\varepsilon-\text{ball}\}.$
  \item $u_{\Omega}$ is the solution of the Dirichlet problem $P(\Omega,f)$.
  \item $\Phi_{h}(\Omega)=\int_{\partial\Omega}h(s)ds$.
  \item $T(\Omega,h)=\int_{\Omega}h$.
  \item $\Psi_{h}(\Omega)=[T(\Omega,u_{\Omega})]^{-1}[\Phi_{\sqrt{h}}(\Omega)]^2$.
  \item $[h]_X=\frac{1}{|X|}\int_X h$.
\end{itemize}
In \cite{d}, the author shows the existence of a minimizer for a functional which is defined on the domain or on its boundary, in the class $\mathcal{O}_{\varepsilon}$ of non-empty open sets $\Omega\subset D$ satisfying both a uniform interior ball and uniform exterior ball conditions.\\
Consider the class of the admissible domains $\mathcal{O}^{\varepsilon}_{C_f}=\mathcal{O}_{\varepsilon}\cap \mathcal{O}_{C_f}$. Then the functionals $\Phi_{g}$ and $\Psi_{g}$ admit a minimizer on $\mathcal{O^{\varepsilon}}_{C_f}$. Let's denote by $\omega_{g}^{C_f}$ (resp. $\Omega_{g}^{C_f}$) a minimizer of $\Phi_{g}$  (resp. $\Psi_{g}$).
\begin{theorem}
$\mathcal{QS}(f,g)$ has a solution if and only if $\Phi_{g}(\omega_{g}^{C_f)}<\int_{C_f}f(x)dx$.
\end{theorem}
%%%%%%%%%%%%%%%%%%%%%%%%%%%%%%%%%%%%%%%%%%%%%%%%%%%%%%%%%%%%%%%%%%%%%%%%%%%%%%%%%%%%%%%%%%%%%%%%%%%%%%%%%%%%%%%%%%%%%%%%%%%%%%%%%%%%%%
\begin{proof}
The necessary condition of this theorem is done by using Green's formula and the definition of $\omega_{g}^{C_f}$. For the sufficient one, we proceed as follows.\\
Consider the functional $$J_{f,g}(\omega)=\int_{\omega}(|\nabla u_{\omega}|^{2}-2fu_{\omega}+g^{2})dx.$$
By the maximum principle, the functional $J$ is lower bounded and by Theorem 2.1, Proposition 2.1 and Theorem 2.2, it admits at least a minimizer $\Omega\in\mathcal{O}^{\varepsilon}_{C_f} $ which is of class $C^{2}$ (it suffices to consider a minimizing sequence $\Omega_{n}$ defined as in Theorem 2.2 and then use the equivalence between convergence in the compact sense and the Hausdorff convergence).
Next, we can use the shape derivative to the functional $J_{f,g}$ at the minimizer $\Omega$ and obtain,
we obtain the following optimality conditions:
\begin{enumerate}
  \item[(C1)] $|\nabla u_{\Omega}|\leq g$ on $\partial\Omega\cap\partial C_f$, and
  \item[(C2)] $|\nabla u_{\Omega}|=g$ on $\partial\Omega\setminus\partial C_f$.
  \end{enumerate}
Indeed, for the part $\partial\Omega\cap\partial C_f$, the only admissible vector fields verify $V.\nu\geq 0$ whereas for $\partial\Omega\setminus\partial C_f$, all fields $V$ are admissible (see Proposition 2.1 above).\quad
Let us suppose by the absurd that $\partial\Omega\cap\partial C_f\neq\emptyset$. Let's put,
$$\mathcal{O}_{\Omega}=\{\omega\subset \Omega,\;\;\omega\in\mathcal{O}^{\varepsilon}_{C_f}\},\;\text{and}$$
$$J_{f,M_{C_f}}(\omega)=\int_{\omega}(|\nabla u_{\omega}|^{2}-2fu_{\omega}+M^2_{C_f})dx,\text{ where } M_{C_f}=(\max_{\partial C_f }g)\frac{\int_{C_f}f(x)dx}{\Phi_{g}(\omega_{g}^{C_f})}.$$
As for $J_{f,g}$, the functional $J_{f,M_{C_f}}$ admits at least a minimizer $\Omega^{*}$ on $\mathcal{O}_{\Omega}$. As in the previous case we can assume without loss of generality that the minimizer $\Omega^{*}$ is of class $C^2$. The shape derivative applied to the functional $J_{f,M_{C_f}}$ at the minimizer $\Omega^{*}$ gives:
\begin{itemize}
  \item[(C3)] $|\nabla u_{\Omega^*}|\leq M_{C_f}$ on $\partial\Omega^{*}\cap\partial C_f$,
  \item[(C4)] $|\nabla u_{\Omega^*}|\geq M_{C_f}$ on $\partial\Omega^{*}\cap\partial\Omega$, and
  \item[(C5)] $|\nabla u_{\Omega^*}|=M_{C_f}$ on $\partial\Omega^{*}\setminus(\partial\Omega\cup\partial C_f)$.
\end{itemize}
Item (C3) is obtained since the only admissible $V$ vector fields verify $V.\nu\geq 0$ while those giving item (C4) must verify $V.\nu\leq 0$. For item (C5), all $V$ vector fields are admissible.\\
Suppose by contradiction that $\partial\Omega\cap\partial C_f\neq\emptyset$. Since $int(C_f)\subset\Omega^{*}\subset\Omega$, one of the following situations occurs:
\begin{enumerate}
  \item $\partial\Omega=\partial\Omega^{*}=\partial C_f$
  \item $\partial\Omega\neq\partial C_f$ and $\partial\Omega^{*}=\partial C_f$
  \item $\partial\Omega\neq\partial C_f$ and $\partial\Omega^{*}\neq\partial C_f$
  \item $\partial\Omega\neq\partial C_f$ and $\partial\Omega^{*}=\partial\Omega$
  \item $\partial\Omega\neq\partial C_f$ and $\partial\Omega^{*}\neq\partial\Omega$
\end{enumerate}
The aim of the sequel is to prove that each of the five cases above contradicts the condition stated in Theorem 3.1., obtaining $\Omega$ as a solution to Problem $\mathcal{QS}(f,g)$. \\
\textbf{Case 1.} $\partial\Omega=\partial\Omega^{*}=\partial C_f$\\
 Since $int(C_f)\subset\Omega^{*}\subset\Omega$ implies that $int(C_f)=\Omega^{*}=\Omega$. Then by (C1),...,(C5)
$$\max_{\partial C_f}g\leq M_{C_f}=|\nabla u_{\Omega^*}(x)|=|\nabla u_{\Omega}(x)|\leq g,\;\forall x\in\partial\Omega.$$
\textbf{Case 2.} $\partial\Omega\neq\partial C_f$ and $\partial\Omega^{*}=\partial C_f$\\
$\partial\Omega^{*}=\partial C_f$ together with $int(C_f)\subset\Omega^{*}$ implies that $int(C_f)=\Omega^{*}$. Then by (C1),...,(C5) and the maximum principle applied to $\Omega^*$ and $\Omega$
$$\max_{\partial C_f}g\leq M_{C_f}=|\nabla u_{\Omega^*}(x)|<|\nabla u_{\Omega}(x)|\leq g ,\;\forall x\in\partial\Omega\cap\partial\Omega^*.$$
\textbf{Case 3.} $\partial\Omega\neq\partial C_f$ and $\partial\Omega^{*}\neq\partial C_f$\\
 By (C1),...,(C5) and the maximum principle applied to $\Omega^*$ and $\Omega$
$$\max_{\partial C_f}g\leq M_{C_f}\leq|\nabla u_{\Omega^*}(x)|<|\nabla u_{\Omega}(x)|\leq g,\;\forall x\in\partial\Omega\cap\partial\Omega^*.$$
\textbf{Case 4.} $\partial\Omega\neq\partial C_f$ and $\partial\Omega^{*}=\partial\Omega$\\
$\partial\Omega^{*}=\partial\Omega$ together with $\Omega^*\subset\Omega$ implies that $\Omega=\Omega^{*}$. Then by (C1),...,(C5)
$$\max_{\partial C}g\leq M_{C_f}\leq|\nabla u_{\Omega^*}(x)|=|\nabla u_{\Omega}(x)|\leq g,\;\forall x\in\partial\Omega\cap\partial C_f.$$
\textbf{Case 5.} $\partial\Omega\neq\partial C_f$ and $\partial\Omega^{*}\neq\partial\Omega$\\
By (C1),...,(C5) and the maximum principle applied to $\Omega^*$ and $\Omega$
$$\max_{\partial C}g\leq M_{C_f}\leq|\nabla u_{\Omega^*}(x)|<|\nabla u_{\Omega}(x)|\leq g,\;\forall x\in\partial\Omega\cap\partial\Omega^*.$$
\end{proof}
\begin{remark}
  If $\int_{\partial C_f}g<\int_{C_f}f$ then $\mathcal{QS}(f,g)$ has a solution.
\end{remark}
%%%%%%%%%%%%%%%%%%%%%%%%%%%%%%%%%%%%%%%%%%%%%%%%%%%%%%%%%%%%%%%%%%%%%%%%%%%%%%%%%%%%%%%%%%%%%%%%%%%%%%%%%%%%%%%%%
\subsection{Existence for Problem $\mathcal{B}(f,g)$}
The purpose here is to prove
\begin{theorem}
$\mathcal{B}(f,g)$ has a solution if and only if $\Psi_{g}(\Omega_{g}^{C_f})<\int_{C_f}f(x)dx$.
\end{theorem}
To prove this theorem, we proceed as follows.\\
Suppose there exists $\Omega$ solution of $\mathcal{B}(f,g)$. Then by Cauchy-Schwarz's inequality
$$\int_{\partial\Omega}\sqrt{g}=\int_{\partial\Omega}\sqrt{|\nabla u_{\Omega}||\nabla v_{\Omega}|}\leq (\int_{\partial\Omega}|\nabla u_{\Omega}|)^{\frac{1}{2}}(\int_{\partial\Omega}|\nabla v_{\Omega}|)^{\frac{1}{2}}.$$
Then Green's formula gives
$$\Psi_{\sqrt{g}}(\Omega)=\frac{[\int_{\partial\Omega}\sqrt{g}]^2}{\int_{\Omega}u_{\Omega}}\leq \int_{\Omega}f(x)dx,$$
and so, $$\Psi_{\sqrt{g}}(\Omega_{g}^{C_f})\leq \int_{C_f}f(x)dx.$$

By using the domain derivative \cite{sz}, the problem $\mathcal{B}(f,g)$
seems to be the Euler equation of the following optimization
problem. Put
$$J_{g}(\omega)=\int_{\omega}g^2dx-\frac{1}{2}\int_{\omega}u^{2}_{\omega}dx.$$
$u_{\omega}$ is the solution of $P(\omega,f)$.

\begin{proposition}
\begin{enumerate}
    \item There exists $\varpi\in\mathcal{O}^{\varepsilon}_{C_f} $ which is of class $C^2$ such that
    $$J_{g}(\varpi)=\text{Min}\{J_{g}(\omega),\;\omega\in \mathcal{O}^{\varepsilon}_{C_f} \}.$$
    \item
    \begin{equation*}
(I) \left\{
\begin{array}{c}
|\nabla u_{\varpi}||\nabla v_{\varpi}|\leq g \;\text{on}\;\partial\varpi\cap\partial C_f\\
|\nabla u_{\varpi}||\nabla v_{\varpi}|= g\;\text{on}\;\partial\varpi\setminus\partial
C_f.
\end{array}
\right.
\end{equation*}
\end{enumerate}
\end{proposition}
Now, put

$$\mu_{C_f}=(\max_{\partial C_f}g)\frac{\int_{C_f}f(x)dx}{\Psi_{\sqrt{g}}(\Omega_{g}^{C_f})},$$
$$J_{\mu_{C_f}}(\omega)=\mu^2_{{C_f}}|\omega|-\frac{1}{2}\int_{\omega}u^{2}_{\omega},\;\; \text{and}$$
$$\mathcal{O}_{\varpi}=\{\omega\subset\varpi,\;\omega\in \mathcal{O}^{\varepsilon}_{C_f} \},$$

\begin{proposition}
\begin{enumerate}
    \item There exists $\varpi^*\in\mathcal{O}_{\Omega}$ such that
    $$J_{\mu_{C_f}}(\varpi^*)=\text{Min}\{J_{\mu_{C_f}}(\omega),\;\omega\in\mathcal{O}_{\varpi}\}.$$
    \item
    \begin{equation*}
(II) \left\{
\begin{array}{c}
|\nabla u_{\varpi^*}||\nabla v_{\varpi^*}|\leq \mu_{C_f}\;\text{on}\;\partial\varpi^*\cap\partial C_f\\
|\nabla u_{\varpi^*}||\nabla v_{\varpi^*}|\geq \mu_{C_f} \;\text{on}\;\partial\varpi^*\cap\partial\varpi\\
|\nabla u_{\varpi^*}||\nabla v_{\varpi^*}|=
\mu_{C_f}\;\text{on}\;\partial\varpi^*\setminus(\partial
C_f\cup\partial\varpi).
\end{array}
\right.
\end{equation*}
\end{enumerate}
\end{proposition}
The proof of the above propositions uses Theorem 2.1, Proposition 2.1, Proposition 2.3 and Lemma 2.1.\\
Next, as in the previous proof, according to $(I)$,
$(II)$ and by applying the maximum principle to $\Omega$ and
$\varpi^*$, We obtain
$$g<\mu_{C_f}=|\nabla u_{\varpi^*}||\nabla v_{\varpi^*}|\leq |\nabla u_{\varpi}||\nabla v_{\varpi}|\leq g,\; \text{on}\; \partial\Omega^{*}\cap\partial\Omega\cap\partial C_f,$$
which is absurd. Obtaining $\varpi$ as a solution to Problem $\mathcal{B}(f,g)$.
\begin{remark}
  If $\Psi_{\sqrt{g}}(C_f)<\int_{C_f}f(x)dx$ then $\mathcal{B}(f,g)$ has a solution.
\end{remark}
\begin{remark}
  Let $C_{1}$ (resp. $C_{2}$) be the convex hull of the support of $f_{1}$ (resp. $f_{2}$). Suppose $f_1\leq f_2$.
  \begin{itemize}
    \item If $\mathcal{QS}(f_1,g)$ has a solution then it is the same for $\mathcal{QS}(f_2,g)$.
    \item By using the maximum principle, if $\mathcal{B}(f_{1},g)$ has a solution then it is the same for $\mathcal{B}(f_{2},g)$.
  \end{itemize}
\end{remark}
%%%%%%%%%%%%%%%%%%%%%%%%%%%%%%%%%%%%%%%%%%%%%%%%%%%%%%%%%%%%%%%%%%%%%%%%%%%%%%%%%%%%%%%%%%%%%%%%%%%%%%%%%%%
\subsection{Some corollaries}
This section concerns some corollaries which we can derived from the two theorems above. The necessary and sufficient conditions presented in Theorems 3.1 and 3.2 seem less practical, therefore we will work with the sufficient condition mentioned in Remarks 3.1 and 3.2.
\begin{proposition}
  Let $f_{k}$ ($k\in[1,n]$) be $n$ positive functions. Put
  $$m_{n}=\min_{1\leq k\leq n}f_{k};\;H_{n}=\frac{n}{\sum_{k=1}^{n}1/f_{k}};\;G_{n}=\sqrt[n]{\prod_{k=1}^{n}f_{k}}$$ $$A_{n}=\frac{1}{n}\sum_{k=1}^{n}f_{k};\;
  Q_{n}=\frac{1}{n}\sum_{k=1}^{n}f^{2}_{k}\text{ and }M_{n}=\max_{1\leq k\leq n}f_{k}.$$
  \begin{enumerate}
    \item If $\mathcal{QS}(m_{n},g)$ has a solution then it is the same for $\mathcal{QS}(H_n,g)$, $\mathcal{QS}(G_n,g)$, $\mathcal{QS}(A_n,g)$, $\mathcal{QS}(Q_n,g)$ and $\mathcal{QS}(M_n,g)$.
    \item If $\mathcal{B}(m_{n},g)$ has a solution then it is the same for $\mathcal{B}(H_n,g)$, $\mathcal{B}(G_n,g)$, $\mathcal{B}(A_n,g)$, $\mathcal{B}(Q_n,g)$ and $\mathcal{B}(M_n,g)$.
  \end{enumerate}
\end{proposition}

The proof of this proposition is an immediate consequence of Remark 3.3.\\
%%%%%%%%%%%%%%%%%%%%%%%%%%%%%%%%%%%%%%%%%%%%%%%%%%%%%%%%%%%%%%%%%%%%%%%%%%%%%%%%%%%%%%%%%%%%%%%%%%%%%%%%%%%%%%%%%%%%%%
As immediate corollaries of Theorems 3.1 and 3.2, we have the following two propositions
\begin{proposition}
 \begin{enumerate}
   \item If $\mathcal{B}(f,g)$ has a solution, it is the same for $\mathcal{QS}(f,[\Phi_{\sqrt{g}}(C_f)][T(C_f,u_{C_f})]^{-1}\sqrt{g})$ and $\mathcal{QS}(u_{C_f},[\Phi_{\sqrt{g}}(C_f)][T(C_f,f)]^{-1}\sqrt{g})$.
   \item If $\mathcal{QS}(f,\sqrt{g})$ and $\mathcal{QS}(u_{C_f},\sqrt{g})$ have (respectively) a solution then it is the same for $\mathcal{S}(f,g)$.
 \end{enumerate}
\end{proposition}
%%%%%%%%%%%%%%%%%%%%%%%%%%%%%%%%%%%%%%%%%%%%%%%%%%%%%%%%%%%%%%%%%%%%%%%%%%%%%%%%%%%%%%%%%%%%%%%%%%%%%%%%%%%%%%

\begin{proposition}
Suppose that in $C_f$, $\phi>0$, $\phi\neq0$ and $\Delta\phi\leq0$. Then,\\
 $\mathcal{QS}(f\phi,|\nabla u_{C_f}|\phi)$ has a solution.
\end{proposition}
\begin{proof}
  By Green's formula
  $$\int_{C_f}f\varphi=\int_{\partial C_f}|\nabla u_{C_f}|\varphi-\int_{C_f}f\Delta\varphi.$$
  As $\Delta\phi\geq0$ on $C_f$, then $$\int_{\partial C_f}|\nabla u_{C_f}|\varphi\leq\int_{C_f}f\varphi,$$
  and the conclusion follows.
\end{proof}
%%%%%%%%%%%%%%%%%%%%%%%%%%%%%%%%%%%%%%%%%%%%%%%%%%%%%%%%%%%%%%%%%%%%%%%%%%%%%%%%%%%%%%%%%%%%%%%%%%%%%%%%%%%%%%%%%%%
\begin{proposition}
  If $\Phi_{\sqrt{g}}(C_f)<\int_{C_f}\sqrt{fu_{C_f}}$, then
\begin{enumerate}
  \item $\mathcal{QS}(\sqrt{fu_{C_f}},\sqrt{g})$ has a solution.
  \item $\mathcal{B}(f,g)$ has a solution.
  \item $\mathcal{QS}(f,[\Phi_g(C_f)][T(C,u_{C_f})]^{-1}\sqrt{g})$ has a solution.
\end{enumerate}
\end{proposition}
  By Cauchy-Schwarz's inequality,
  $$[\Phi_{\sqrt{g}}(C_f)]^{2}\leq[\int_{C_f}\sqrt{fu_{C_f}}]^{2}\leq\int_{C_f} f\int_{C_f} u_{C_f}.$$
  Then, the three item follow.

%%%%%%%%%%%%%%%%%%%%%%%%%%%%%%%%%%%%%%%%%%%%%%%%%%%%%%%%%%%%%%%%%%%%%%%%%%%%%%%%%%%%%%%%%%%%%%%%%%%%%%%%%%%%
\begin{proposition}
  Suppose that $|\nabla u_{C_f}||\nabla v_{C_f}|\geq g\;\text{on }\partial C_f$. Then
  \begin{itemize}
    \item either $\mathcal{QS}(f,[\Phi_{\sqrt{g}}(C_f)][T(C_f,u_{C_f})]^{-1}\sqrt{g})$ and $\mathcal{QS}(u_{C_f},[\Phi_{\sqrt{g}}][T(C_f,f)]^{-1}\sqrt{g})$ have a solution
    \item or $|\nabla u_{C_f}||\nabla v_{C_f}|=g\;\text{on }\partial C_f$.
  \end{itemize}
\end{proposition}
\begin{proof}
  By Cauchy-Schwarz's inequality together with Green's formula applied to $u_{C_f}$ and $v_{C_f}$,
  $$[\Phi_{\sqrt{g}}(C_f)]\int_{\partial C_f}\sqrt{g}\leq[\int_{\partial C_f}\sqrt{|\nabla u_{C_f}||\nabla v_{C_f}|}]^{2}\leq\int_{C_f} f\int_{C_f} u_{C_f}=[T(C,u_{C_f})]\int_{C_f} f.$$
  Then either $\mathcal{QS}(f,[\Phi_{\sqrt{g}}(C_f)][T(C,u_{C_f})]^{-1}\sqrt{g})$ and $\mathcal{QS}(u_{C_f},[\Phi_{\sqrt{g}}][T(C_f,f)]^{-1}\sqrt{g})$ have a solution or $|\nabla u_C||\nabla v_C|=g\;\text{on }\partial C_f$.
\end{proof}
%%%%%%%%%%%%%%%%%%%%%%%%%%%%%%%%%%%%%%%%%%%%%%%%%%%%%%%%%%%%%%%%%%%%%%%%%%%%%%%%%%%%%%%%%%%%%%%%%%%%%%%%%%%%%

\begin{proposition}
  Suppose that $|\nabla u_{C_f}|\geq g\;\text{on }\partial C_f$. Then
  \begin{itemize}
    \item either $\mathcal{QS}(f,[\sqrt{g}]_{\partial C_f}\sqrt{g)}$ has a solution
    \item or $|\nabla u_{C_f}|=g\;\text{on }\partial C_f$.
  \end{itemize}
\end{proposition}
\begin{proof}
  Once again, by Cauchy-Schwarz's inequality together with Green's formula applied to $u_C$,
   $$[\int_{\partial C_f}\sqrt{g}]^2\leq[\int_{\partial C_f}\sqrt{|\nabla u_{C_f}|}]^{2}\leq|\partial C_f|\int_{C_f} f.$$
   Then $$[\sqrt{g}]_{\partial C_f}\int_{\partial C_f}\sqrt{g}\leq\int_{C_f} f,$$
   and the conclusion follows.
\end{proof}
%%%%%%%%%%%%%%%%%%%%%%%%%%%%%%%%%%%%%%%%%%%%%%%%%%%%%%%%%%%%%%%%%%%%%%%%%%%%%%%%%%%%%%%%%%%%%%%%%%%%%%%%%%%%%%%
\begin{proposition}
  Suppose that $|\nabla v_{C_f}|\geq g\;\text{on }\partial C_f$. Then
  \begin{itemize}
    \item either $\mathcal{QS}(u_{C_f},[\sqrt{g}]_{\partial C_f}\sqrt{g)}$ has a solution
    \item or $|\nabla v_{C_f}|=g\;\text{on }\partial C_f$.
  \end{itemize}
\end{proposition}
\begin{proof}
  Using the Cauchy-Schwarz's inequality together with Green's formula applied to $v_{C_f}$,
   $$[\int_{\partial C_f}\sqrt{g}]^2\leq[\int_{\partial C_f}\sqrt{|\nabla v_{C_f}|}]^{2}\leq|\partial C_f|\int_{C_f} u_{C_f}.$$
   Then $$[\sqrt{g}]_{\partial C_f}\int_{\partial C_f}\sqrt{g}\leq\int_{C_f} u_{C_f},$$
   and we obtain item 1. and item 2..
\end{proof}
%%%%%%%%%%%%%%%%%%%%%%%%%%%%%%%%%%%%%%%%%%%%%%%%%%%%%%%%%%%%%%%%%%%%%%%%%%%%%%%%%%%%%%%%%%%%%%%%%%%%%%%%%%%
Now, Using successively Green's formula, we obtain the following two propositions.
In fact, in the first iteration, Green's formula gives,
$$\int_{\Omega}f=\int_{\partial \Omega}|\nabla u_0|f-\int_{\Omega}u_0\Delta f.$$
then,
$$\int_{\Omega}u_0\Delta f=\int_{\partial \Omega}|\nabla u_1|f-\int_{\Omega}u_1\Delta^2 f.$$
Thus after $n$ iterations, we obtain the result mentioned in our propositions.
%%%%%%%%%%%%%%%%%%%%%%%%%%%%%%%%%%%%%%%%%%%%%%%%%%%%%%%%%%%%%%%%%%%%%%%%%%%%%%%%%%%%%%%%%%%%%%%%%%%%%%%%%%%%%
\begin{proposition}
  Let $u_0$ be the solution of $P(\Omega,1)$. Let $u_k$ ($k\geq 1$) be the solution of $P(\Omega,u_{k-1})$. Let $n$ be the smallest integer such that $\Delta^n f\neq0$ in $\Omega$. Suppose that for $1\leq k\leq n$. $(-1)^k\Delta^k f\leq0$ on $\partial\Omega$.
  \begin{enumerate}
    \item If $n$ is even and $\Delta^n f\geq0$ or if $n$ is odd and $\Delta^n f\leq0$ in $\Omega$ then $\mathcal{QS}(f,g)$ has a solution which strictly contains $\text{conv}(\Omega)$ where
  $$g=\sum_{k=0}^{n-1}(-1)^k\frac{\partial u_k}{\partial\nu}\Delta^k f.$$
    \item If $(-1)^{n-1}\Delta^{n-1} f\geq 0$ on $\partial \Omega$ then $\mathcal{QS}(f-(-1)^{n}\Delta^{n} f,g)$ has a solution which strictly contains $\text{conv}(\Omega)$ where
  $$g=\sum_{k=0}^{n-2}(-1)^k\frac{\partial u_k}{\partial\nu}\Delta^k f.$$
  \end{enumerate}
  \end{proposition}
  %%%%%%%%%%%%%%%%%%%%%%%%%%%%%%%%%%%%%%%%%%%%%%%%%%%%%%%%%%%%%%%%%%%%%%%%%%%%%%%%%%%%%%%%%%%%%%%%%%%%%%%%%%%%%%%%%%%%%%%%%%%%%%%%%%%%%%
\begin{proposition}
  Let $f$ and $u_{k}$ ($1\leq k\leq n$) as in the previous proposition. Then
  \begin{enumerate}
    \item either $\mathcal{QS}(f+(-1)^n(\Delta^n f)u_{n},g)$ has a solution which strictly contains $\text{conv}(\Omega)$ where
  $$g=(n-1)\sqrt[n-1]{\prod_{k=0}^{n-1}|\nabla u_{k}|(-1)^k\Delta^k f}.$$
    \item or $(-1)^l(\Delta^l f)|\nabla u_{l}|=(-1)^k(\Delta^k f)|\nabla u_{k}|$ on $\partial\Omega$ ($l\neq k$, $k,\;l\in[0,n-1])$.
  \end{enumerate}
  \end{proposition}
  %%%%%%%%%%%%%%%%%%%%%%%%%%%%%%%%%%%%%%%%%%%%%%%%%%%%%%%%%%%%%%%%%%%%%%%%%%%%%%%%%%%%%%%%%%%%%%%%%%%%%%%%%%%%%%%%%%%%%%%%%%%%%%%%%%%%%%
\begin{remark}
   The proof of this proposition uses also the fact that the geometric mean is less than the arithmetic one for the $n$ positive terms $(-1)^k(\Delta^k f)|\nabla u_{k}|$ for $k,\;l\in[0,n-1])$.
\end{remark}
%%%%%%%%%%%%%%%%%%%%%%%%%%%%%%%%%%%%%%%%%%%%%%%%%%%%%%%%%%%%%%%%%%%%%%%%%%%%%%%%%%%%%%%%%%%%%%%%%%%%%%%%%%%%%
\section{Use of some classical inequalities}
Up to now, we will use some classical inequalities to give solutions for several overdetermined free boundary problems.

\subsection{Use of Cauchy-Schwarz's inequality}
The purpose of this section is the use Cauchy-Schwarz's inequality to get a solution to Problem $\mathcal{QS}(f,g)$.
%%%%%%%%%%%%%%%%%%%%%%%%%%%%%%%%%%%%%%%%%%%%%%%%%%%%%%%%%%%%%%%%%%%%%%%%%%%%%%%%%%%%%%%%%%%%%%%%%%%%%%%%%%%%%%%%%%%%%%%%%%
\begin{proposition} Suppose $g\leq|\nabla v_{C_f}|$ on $\partial C_f$. Put $\gamma=[\Phi_{g}(C_f)][T(C_f,u_{C_f}/v_{C_f})]^{-1}$ and $\delta=[\Phi_{g}(C_f)][T(C_f,u_{C_f}v_{C_f})]^{-1}$. Then
\begin{enumerate}
  \item either $\mathcal{QS}(u_{C_f}v_{C_f},\gamma g)$ and $\mathcal{QS}(u_{C_f}/v_{C_f},\delta g)$ have a solution
  \item or $|\nabla v_{C_f}|=g$ on $\partial C_f$.
\end{enumerate}
 \end{proposition}
 \begin{proof}
   The Cauchy-Schwarz's inequality together with Green's formula applied to $v_{C_f}$,
   $$\int_{\partial C_f}g\leq\int_{\partial C_f} |\nabla v_{C_f}|=\int_{C_f}u_{C_f}=\int_{C_f}u_{C_f}\frac{\sqrt{v_{C_f}}}{\sqrt{v_{C_f}}}\leq[\int_{C_f}u_{C_f}v_{C_f}]^{(1/2)}[\int_{C_f}u_{C_f}/v_{C_f}]^{(1/2)}.$$
   Thus, the conclusion follows.
 \end{proof}
%%%%%%%%%%%%%%%%%%%%%%%%%%%%%%%%%%%%%%%%%%%%%%%%%%%%%%%%%%%%%%%%%%%%%%%%%%%%%%%%%%%%%%%%%%%%%%%%%%%%%%%%%%%%%%%%%%%%%%%%%%%
For the three following propositions, we take $u_{C_1}$ as the solution of $P(C_1,1)$. The integral inequality we get by combining the Cauchy-Schwarz's inequality and the Green's formula allows us to prove that either we have a solution to some quadrature surfaces problem or $C_1$ is a ball.
\begin{proposition} Put $\gamma=[1/|\nabla v_{C_1}|]_{\partial C_1}^{-1}$. Then
\begin{enumerate}
  \item either $\mathcal{QS}(u_{C_1},\gamma)$ has a solution
  \item or $C_1$ is a ball ($|\nabla v_{C_1}|=constant$ on $\partial C_1$).
 \end{enumerate}
\end{proposition}
\begin{proof}
   The Cauchy-Schwarz's inequality together with Green's formula applied to $v_{C_1}$ gives,
   $$|\partial C_1|\leq\int_{\partial C_1}\sqrt{|\nabla v_{C_1}|}\frac{1}{\sqrt{|\nabla v_{C_1}|}}\leq[\int_{\partial C_1}|\nabla v_{C_1}|]^{1/2}[\int_{\partial C_1}\frac{1}{\sqrt{|\nabla v_{C_1}|}}]^{1/2}.$$
   Thus, the conclusion follows. The second item is due to Theorem 2.3.
 \end{proof}
%%%%%%%%%%%%%%%%%%%%%%%%%%%%%%%%%%%%%%%%%%%%%%%%%%%%%%%%%%%%%%%%%%%%%%%%%%%%%%%%%%%%%%%%%%%%%%%%%%%%%%%%%%%%%%%%%%%%%%%%%%%%
%%%%%%%%%%%%%%%%%%%%%%%%%%%%%%%%%%%%%%%%%%%%%%%%%%%%%%%%%%%%%%%%%%%%%%%%%%%%%%%%%%%%%%%%%%%%%%%%%%%%%%%%%%%%%%%%%%%%
\begin{proposition} Suppose $g\leq\sqrt{|\nabla u_{C_1}||\nabla v_{C_1}|}$. Put $\gamma=[\Phi_{\sqrt{g}}(C_1)]/|C_1|$. Then
\begin{enumerate}
  \item either $\mathcal{QS}(u_{C_1},\gamma g)$ has a solution
  \item or $C_1$ is a ball ($|\nabla v_{C_1}|=c|\nabla u_{C_1}|$ on $\partial C_1$). In that case $g=c|\nabla u_{C_1}|^2$
\end{enumerate}
\end{proposition}
\begin{proof}
   By Cauchy-Schwarz's inequality together with Green's formula applied to $u_C$ and $v_C$,
  $$[\Phi_{\sqrt{g}}(C_1)]\int_{\partial C_1}\sqrt{g}\leq[\int_{\partial C_1}\sqrt{|\nabla u_{C_1}||\nabla v_{C_1}|}]^{2}\leq|C_1|\int_{C_1} u_{C_1}.$$
  Then either $\mathcal{QS}(u_{C_1},\gamma g)$ has a solution or $|\nabla v_{C_1}|=c|\nabla u_{C_1}|$ on $\partial C_1$ which means that $C_1$ is a ball according to Theorem 2.3.
 \end{proof}
%%%%%%%%%%%%%%%%%%%%%%%%%%%%%%%%%%%%%%%%%%%%%%%%%%%%%%%%%%%%%%%%%%%%%%%%%%%%%%%%%%%%%%%%%%%%%%%%%%%%%%%%%%%%%

%%%%%%%%%%%%%%%%%%%%%%%%%%%%%%%%%%%%%%%%%%%%%%%%%%%%%%%%%%%%%%%%%%%%%%%%%%%%%%%%%%%%%%%%%%%%%%%%%%%%%%%%%%%%%%%%%%%%%%
\begin{proposition} Let $C$ be a convex set of $D$.
Let $\lambda_{1}(C)$ be the first eigenvalue for the Laplacian. Suppose that $|\nabla u_C|\geq g\; \text{on}\; \partial C$. Then
\begin{enumerate}
  \item either $\mathcal{QS}(1/u_C,\lambda_{1}(C)g)$ has a solution,
  \item or $|\nabla u_C|=g\; \text{on}\; \partial C$.
\end{enumerate}

\end{proposition}
\begin{proof}
  According to Proposition 2.3 in \cite{val}, $$\lambda_{1}(C)\int_{C}u_C\leq|C|.$$
  Now, by Cauchy-Schwarz's inequality, $$|C|\leq\frac{1}{|C|}\int_{C}u_C\int_{C}1/u_C.$$
 Then by Green's formula $$\int_{\partial C}|\nabla u_C|=|C|\leq\frac{1}{\lambda_{1}(C)}\int_{C}1/u_C,$$
 which gives the conclusion.
\end{proof}
%%%%%%%%%%%%%%%%%%%%%%%%%%%%%%%%%%%%%%%%%%%%%%%%%%%%%%%%%%%%%%%%%%%%%%%%%%%%%%%%%%%%%%%%%%%%%%%%%%%%%%%%%%%%%%%%%%%%%%%%%%%%%%%%%%%%%%
%%%%%%%%%%%%%%%%%%%%%%%%%%%%%%%%%%%%%%%%%%%%%%%%%%%%%%%%%%%%%%%%%%%%%%%%%%%%%%%%%%%%%%%%%%%%%%%%%%%%%%%%%%%%%%%%%%%%%%%%%%%%%%%%%%%%%%%
\subsection{Use of eigenvalue inequalities}
Many of inequalities we will use here represent those among the expressions
$$\int_{B} u^2dx,\;\; \int_{\partial B}u^2ds,\;\;\int_{B} |\nabla u|^2dx,\;\;\int_{\partial B}(\frac{\partial u}{\partial\nu})^2ds,\;\;\int_{B}(\Delta u)^2dx,$$
for functions $u$ satisfying various smoothness and auxiliary conditions. All the inequalities have the common property that they arise from variational characterization of various eigenvalues (see e.g. \cite{si}). They are

$$\int_{B} u^2dx\leq \lambda^{-1}\int_{B} |\nabla u|^2dx,\;\; u=0\;\; on\;\partial B$$
$$\int_{B} u^2dx\leq \mu^{-1}\int_{B} |\nabla u|^2dx,\;\; \int_{B}udx=0\;\; on\;\partial B,$$
$$\int_{B} u^2dx\leq \rho^{-1}\int_{B} (\Delta u)^2dx,\;\; u=\frac{\partial u}{\partial\nu}=0\;\; on\;\partial B,$$
$$\int_{B} u^2dx\leq \lambda^{-2}\int_{B} (\Delta u)^2dx,\;\; u=0\;\; on\;\partial B,$$
$$\int_{B} u^2\leq \mu^{-2}\int_{B} (\Delta u)^2dx,\;\; \frac{\partial u}{\partial\nu}=0\;\; on\;\partial B,$$
$$\int_{\partial B}u^2ds\leq p^{-1}\int_{B} |\nabla u|^2dx,\;\;\int_{\partial B}uds=0,$$
$$\int_{\partial B}u^2ds\leq \xi^{-1}\int_{B}(\Delta u)^2dx,\;\;\frac{\partial u}{\partial\nu}=0\;\;\int_{\partial B}uds=0,$$
$$\int_{B} |\nabla u|^2dx\leq \Lambda^{-1}\int_{B}(\Delta u)^2dx,\;\;u=\frac{\partial u}{\partial\nu}=0\; on\;\partial B,$$
$$\int_{B} |\nabla u|^2dx\leq \lambda^{-1}\int_{B}(\Delta u)^2dx,\;\;u=0\;\; on\;\partial B,$$
$$\int_{B} |\nabla u|^2dx\leq \mu^{-1}\int_{B}(\Delta u)^2dx,\;\;\frac{\partial u}{\partial\nu}=0\;\; on\;\partial B,$$
$$\int_{\partial B}(\frac{\partial u}{\partial\nu})^2ds\leq q^{-1}\int_{B}(\Delta u)^2dx,\;\;u=0\;\; on\;\partial B.$$
The optimal constants in the above inequalities are the reciprocals of the first non-zero eigenvalues of the following problems:
\begin{enumerate}
  \item The fixed membrane $$\Delta u+\lambda u=0\;\;on\;\;B,\;\; u=0\;\;on\;\;\partial B.$$
  \item The free membrane $$\Delta u+\mu u=0\;\;on\;\;B,\;\; \frac{\partial u}{\partial\nu}=0\;\;on\;\;\partial B.$$
  \item The clamped plate $$\Delta^2 u-\rho u=0\;\;on\;\;B,\;\; u=\frac{\partial u}{\partial\nu}=0\;\;on\;\;\partial B.$$
  \item The buckling of the clamped plate $$\Delta^2 u-\Lambda\Delta u=0\;\;on\;\;B,\;\; u=\frac{\partial u}{\partial\nu}=0\;\;on\;\;\partial B.$$
  \item The Stekloff problems
  $$\Delta^2 u=0\;\;on\;\;B,\;\;\frac{\partial u}{\partial\nu}=pu\;\;on\;\;\partial B.$$
  $$\Delta u=0\;\;on\;\;B,\;\;\frac{\partial u}{\partial\nu}=\frac{\partial\Delta u}{\partial\nu}+\xi u=0\;\;on\;\;\partial B.$$
  $$\Delta u=0\;\;on\;\;B,\;\;u=\Delta u-q\frac{\partial u}{\partial\nu}=0\;\;on\;\;\partial B.$$
\end{enumerate}
We denote the optimal constants in the above inequalities by $\lambda_1,\;\mu_2,\;p_2,\;\rho_1,\;\Lambda_1,\;q_1$ and $\xi_2$. Now we can state
\begin{theorem}
Let $\lambda,\;\mu,\;p,\;\rho,\;\Lambda,\;q$ and $\xi$ be respectively different to $\lambda_1,\;\mu_2,\;p_2,\;\rho_1,\;\Lambda_1,\;q_1$ and $\xi_2$.
\begin{enumerate}
  \item If $u=0$ on $\partial B$ and $\mathcal{QS}(u^2,g)$ has a solution which strictly contain $\overline{B}$, then it is the same for $\mathcal{QS}(|\nabla u|^2,\lambda g)$.
  \item If $\int_ B u=0$ and $\mathcal{QS}(u^2,g)$ has a solution which strictly contain $\overline{B}$, then it is the same for $\mathcal{QS}(|\nabla u|^2,\mu g)$.
  \item If $u=\frac{\partial u}{\partial\nu}=0$ on $\partial B$ and $\mathcal{QS}(u^2,g)$ has a solution which strictly contain $\overline{B}$, then it is the same for $\mathcal{QS}((\Delta u)^2,\rho g)$.
  \item If $u=0$ on $\partial B$ and $\mathcal{QS}(u^2,g)$ has a solution which strictly contain $\overline{B}$, then it is the same for $\mathcal{QS}((\Delta u)^2,\lambda^2 g)$.
  \item If $\frac{\partial u}{\partial\nu}=0$ on $\partial B$ and $\mathcal{QS}(u^2,g)$ has a solution which strictly contain $\overline{B}$, then it is the same for $\mathcal{QS}((\Delta u)^2,\mu^2 g)$.
  \item If $\int_ B u=0$, then $\mathcal{QS}(|\nabla u|^2,pu^2)$ has a solution which strictly contain $\overline{B}$
  \item If $\int_ B u=0$ and $\frac{\partial u}{\partial\nu}=0$ on $\partial B$ and , then $\mathcal{QS}((\Delta u)^2,\xi u^2)$ has a solution which strictly contain $\overline{B}$
  \item If $u=\frac{\partial u}{\partial\nu}=0$ on $\partial B$ and $\mathcal{QS}(|\nabla u|^2,g)$ has a solution which strictly contain $\overline{B}$, then it is the same for $\mathcal{QS}((\Delta u)^2,\Lambda g)$.
  \item If $u=0$ on $\partial B$ and $\mathcal{QS}(|\nabla u|^2,g)$ has a solution which strictly contain $\overline{B}$, then it is the same for $\mathcal{QS}((\Delta u)^2,\lambda g)$.
  \item If $\frac{\partial u}{\partial\nu}=0$ on $\partial B$ and $\mathcal{QS}(|\nabla u|^2,g)$ has a solution which strictly contain $\overline{B}$, then it is the same for $\mathcal{QS}((\Delta u)^2,\mu g)$.
  \item If $\int_ B u=0$, then $\mathcal{QS}((\Delta u)^2,q(\frac{\partial u}{\partial\nu})^2)$ has a solution which strictly contain $\overline{B}$
\end{enumerate}
\end{theorem}
%%%%%%%%%%%%%%%%%%%%%%%%%%%%%%%%%%%%%%%%%%%%%%%%%%%%%%%%%%%%%%%%%%%%%%%%%%%%%%%%%%%%%%%%%%%%%%%%%%%%%%%%%%%
\begin{proposition}
Let $u\in H^{1,p}_{0}(\Omega)$ ($p>1$) and let $d$ be the distance from $\partial\Omega$. If $\mathcal{QS}(|\frac{u}{d}|^p,g)$ has a solution which strictly contains $\overline{\Omega}$, then for $c\geq(\frac{p}{p-1})^p$, $\mathcal{QS}(c|\nabla u|^p,g)$ has a solution which strictly contains $\overline{\Omega}$.
\end{proposition}
The proof of this proposition uses Hardy's inequality: If $u\in H^{1,p}_{0}(\Omega)$ ($p>1$), then for $c\geq(\frac{p}{p-1})^p$
$$\int_\Omega |\frac{u}{d}|^p\leq c\int_\Omega|\nabla u|^p.$$

%%%%%%%%%%%%%%%%%%%%%%%%%%%%%%%%%%%%%%%%%%%%%%%%%%%%%%%%%%%%%%%%%%%%%%%%%%%%%%%%%%%%%%%%%%%%%%%%%%%%%%%%%%%%%
Now, if $\Omega$ is starshaped ($x.\nu>0$ on $\partial\Omega$) then Pohozaev's identity gives for $N\geq3$
$$\int_{\Omega}|\nabla u|^2\leq\frac{2N}{N-2}\int_{\Omega}f\circ u.$$
The conclusion of the following proposition follows.
\begin{proposition}
Let $\Omega$ be star-shaped and let $u$ be such that $-\triangle u=f'(u)$. If $\mathcal{QS}(|\nabla u|^{2},g)$ has a solution which strictly contains $\text{conv}(\Omega)$, then for $N\geq 3$, $\mathcal{QS}(\frac{2N}{N-2}| f\circ u,g)$ has a solution which strictly contains $\text{conv}(\Omega)$.
\end{proposition}
%%%%%%%%%%%%%%%%%%%%%%%%%%%%%%%%%%%%%%%%%%%%%%%%%%%%%%%%%%%%%%%%%%%%%%%%%%%%%%%%%%%%%%%%%%%%%%%%%%%%%%%%%%%%%%%%%%%%%%%%%%%%%%%%%%%%%%%%
Let $u_{\Omega}$ (resp. $v_{\Omega}$) be the solution of $P(\Omega,1)$ (resp. $P(\Omega,u_{\Omega}$)). Using Pohozaev's identity for both $u_{\Omega}$ and $v_{\Omega}$, we obtain
$$\int_{\partial\Omega}|\nabla u_{\Omega}|^2x.\nu=(2+N)\int_{\Omega}u_{\Omega},\text{ and}$$
$$\int_{\partial\Omega}|\nabla v_{\Omega}|^2x.\nu=(2+N)\int_{\Omega}u_{\Omega}v_{\Omega}.$$
\begin{proposition}
  Let $\Omega$ be starshaped and suppose $|\nabla u_{\Omega}||\nabla v_{\Omega}|x.\nu\geq g$ on $\partial\Omega$. Then
  \begin{enumerate}
    \item either $\mathcal{QS}(u_{\Omega},\frac{1}{(2+N)^2}\Phi_g(\Omega))[T(\Omega,u_{\Omega}v_{\Omega}]^{-1}g)$ and $\mathcal{QS}(u_{\Omega}v_{\Omega},\frac{1}{(2+N)^2}\Phi_g(\Omega))[T(\Omega,u_{\Omega}]^{-1}g)$ have a solution
    \item or $\Omega$ is a ball ($|\nabla v_{\Omega}|=\lambda|\nabla u_{\Omega}|$ on $\partial\Omega$).
  \end{enumerate}
\end{proposition}
We conclude by Cauchy-Schwarz's inequality,
$$(\int_{\partial\Omega}g)^2\leq(\int_{\partial\Omega}|\nabla u_{\Omega}||\nabla v_{\Omega}|x.\nu)^2\leq(\int_{\partial\Omega}|\nabla u_{\Omega}|x.\nu)(\int_{\partial\Omega}|\nabla u_{\Omega}|x.\nu)=(2+N)^2\int_{\Omega}u_{\Omega}\int_{\Omega}u_{\Omega}v_{\Omega}.$$
%%%%%%%%%%%%%%%%%%%%%%%%%%%%%%%%%%%%%%%%%%%%%%%%%%%%%%%%%%%%%%%%%%%%%%%%%%%%%%%%%%%%%%%%%%%%%%%%%%%%%%%%%%%%%%%%%%%%%%%%%%%%%%%%%%%%%%%%
\section{Final remarks}
\begin{remark}
  \begin{proposition}
  If $\mathcal{QS}(f,g)$ has a solution then it is the same for $\mathcal{QS}(f^{2n},(\frac{\int_{\partial\Omega_g}g}{|\Omega_g|})^{2^n-1}g)$ for all $n\geq 1$.
\end{proposition}
The proof of this proposition is done by performing the Cauchy-Schwarz's inequality $n$ times.\\
\end{remark}
%%%%%%%%%%%%%%%%%%%%%%%%%%%%%%%%%%%%%%%%%%%%%%%%%%%%%%%%%%%%%%%%%%%%%%%%%%%%%%%%%%%%%%%%%%%%%%%%%%%%%%%%%%%%%
\begin{remark}
  Consider the following functional introduced in \cite{hc}:
$$J(\Omega)=N\int_{\partial\Omega}|\nabla u_{\Omega}|^{3}d\sigma-(N+2)\int_{\Omega}|\nabla u_{\Omega}|^{2}dx.$$
Let $u_{\Omega}$ (resp. $v_{\Omega}$) be the solution of $P(\Omega,1)$ (resp. $P(\Omega,u_{\omega})$). As in \cite{hc},
$$\int_{\Omega}(|\nabla u_{\Omega}|^{2}+\Delta(\nabla u_{\Omega}|^{2})u_{\Omega})dx\geq\int_{\partial\Omega}(\frac{N+2}{N})|\nabla v_{\Omega}|d\sigma.$$
\begin{proposition}
  \begin{enumerate}
    \item Either $\mathcal{QS}(|\nabla u_{\Omega}|^{2}+\Delta(\nabla u_{\Omega}|^{2})u_{\Omega},\frac{N+2}{N}|\nabla v_{\Omega}|)$.
    \item Or $\Omega$ is a ball ($J(\Omega)=0$).
  \end{enumerate}
\end{proposition}
\end{remark}
\begin{remark}
  Let $u_0$ be the solution of $P(\Omega,1)$. Let $u_k$ ($k\geq 1$) be the solution of $P(\Omega,u_{k-1})$. Then
$$\int_{\partial\Omega}\sqrt[n]{\prod_{k=0}^{n}|\nabla u_{k}|}\leq\frac{1}{n}\sum_{k=0}^{n-1}\int_{\partial\Omega}|\nabla u_{k}|,$$ by Green's formula
$$\int_{\partial\Omega}\sqrt[n]{\prod_{k=0}^{n}|\nabla u_{k}|}\leq\frac{1}{n}(\int_{\Omega}(1+\sum_{k=0}^{n-1}u_{k-1}).$$
If $\sqrt[n]{\prod_{k=0}^{n}|\nabla u_{k}|}\geq g$ on $\partial\Omega$, then
\begin{enumerate}
    \item either $\mathcal{QS}(1+\sum_{k=1}^{n-1}u_{k-1},g)$ has a solution which strictly contains $\text{conv}(\Omega)$
    \item or $|\nabla u_{l}|=|\nabla u_{k}|=g$ on $\partial\Omega$ ($l\neq k$, $k,\;l\in[0,n-1])$.
  \end{enumerate}
\end{remark}
%%%%%%%%%%%%%%%%%%%%%%%%%%%%%%%%%%%%%%%%%%%%%%%%%%%%%%%%%%%%%%%%%%%%%%%%%%%%%%%%%%%%%%%%%%%%%%%%%%%%%%%%%%%%%
\begin{remark}
In 1901, Minkowski proved that the following inequality holds for any non-empty open convex set $\Omega$ of $\mathbb{R}^3$ whose boundary is $C^2$-surface,
$$\int_{\partial\Omega}H_{\partial\Omega}\geq 2\sqrt{4\pi|\partial\Omega|}.$$
By using Cauchy-Schwarz's inequality, we get
\begin{enumerate}
      \item either there exists $\Omega^*$ containing strictly $\bar{\Omega}$ which solves $\mathcal{QS}(f,2\sqrt{\frac{4\pi}{|\partial\Omega|}})$
      \item or $\Omega$ is a ball and $|\nabla u_{\Omega}|=H_{\partial\Omega}$ on $\partial\Omega$ ($u_{\Omega}$ being the solution of $P(\Omega,f)$)
    \end{enumerate}
\end{remark}
%%%%%%%%%%%%%%%%%%%%%%%%%%%%%%%%%%%%%%%%%%%%%%%%%%%%%%%%%%%%%%%%%%%%%%%%%%%%%%%%%%%%%%%%%%%%%%%%%%%%%%%%%%%%%%%%%
\begin{remark}
In \cite{bal}, Bucur et al.  proved that the following inequality holds for any non-empty open convex set $\Omega$ of $\mathbb{R}^3$
$$\int_{\partial\Omega}|x|^{2}\geq (\frac{|\Omega|}{|B_{N}|})^{2/N}|\partial\Omega|.$$
Let $u_{\Omega}$ be the solution of $P(\Omega,f)$. If $|\nabla u_{\Omega}|\geq |x|^2$ on $\partial\Omega$, then
\begin{enumerate}
      \item either there exists $\Omega^*$ containing strictly $\bar{\Omega}$ solution of $\mathcal{QS}(f,(\frac{|\Omega|}{|B_{N}|})^{2/N})$
      \item or $\Omega$ is a ball.
    \end{enumerate}
\end{remark}
%%%%%%%%%%%%%%%%%%%%%%%%%%%%%%%%%%%%%%%%%%%%%%%%%%%%%%%%%%%%%%%%%%%%%%%%%%%%%%%%%%%%%%%%%%%%%%%%%%%%%%%%%%%%%%%%%%
Let $u_{\Omega}$ (resp. $v_{\Omega}$) be the solution of $P(\Omega,1)$ (resp. $P(\Omega,u_{\Omega})$. Let $r$ (resp. $R$) be the radius of the biggest ball included in $\Omega$ (resp. the radius of the smallest ball containing $\Omega$). Then $r\leq|\nabla u_{\Omega}|\leq c_N\frac{d_\Omega(d_\Omega+R)}{R}$.
\begin{remark}
  Let $u_{\Omega}$ (resp. $v_{\Omega}$) be the solution of $P(\Omega,1)$ (resp. $P(\Omega,u_{\Omega})$. Let $r$ (resp. $R$) be the radius of the biggest ball included in $\Omega$ (resp. the radius of the smallest ball containing $\Omega$). Then $r\leq|\nabla u_{\Omega}|\leq c_N\frac{d_\Omega(d_\Omega+R)}{R}$. $d_\Omega$ being the diameter of $\Omega$.
\begin{proposition} If $|\nabla u_{\Omega}||\nabla v_{\Omega}|\geq g$ on $\partial\Omega$, then
\begin{enumerate}
  \item either $\mathcal{QS}(u_{\Omega},\frac{R}{c_{N}d_{\Omega}(d_{\Omega}+R)}g)$
  \item or $\Omega$ is a ball ($|\nabla u_{\Omega}|=c_N\frac{d_\Omega(d_\Omega+R)}{R}$ on $\partial\Omega$) and $|\nabla v_{\Omega}|=\frac{R}{c_Nd_\Omega(d_\Omega+R)}g$ on $\partial\Omega$.
\end{enumerate}
\end{proposition}
\end{remark}
%%%%%%%%%%%%%%%%%%%%%%%%%%%%%%%%%%%%%%%%%%%%%%%%%%%%%%%%%%%%%%%%%%%%%%%%%%%%%%%%%%%%%%%%%%%%%%%%%%%%%%%%%%%%%%
\begin{remark}
  In \cite{cm}, the authors gave the following inequality based on Relly's identity: for $u_{\Omega}$ solution of $P(\Omega,1)$
  $$(\int_{\partial\Omega}|\nabla u_{\Omega}|)^2\leq\int_{\partial\Omega}1/H_{\Omega}\int_{\Omega}(1-|\nabla u_{\Omega}|^2).$$
  We can deduce the following
  \begin{proposition}
    If $|\nabla u_{\Omega}|\geq g$ on $\partial\Omega$, then
    \begin{enumerate}
      \item either $\mathcal{QS}(1-|\nabla u_{\Omega}|^2,\frac{\int_{\partial\Omega}|\nabla u_{\Omega}|}{\int_{\partial\Omega}1/H_{\Omega}})g)$
      \item or $\Omega$ is a ball ($|\nabla u_{\Omega}|=1/H_{\Omega}$ on $\partial\Omega$).
    \end{enumerate}
  \end{proposition}
  The second item is due to Lemma 2.2 above.
\end{remark}
By Pohozaev's identity, $$(N+2)\int_{\Omega}u_{\Omega}=\int_{\partial\Omega}|\nabla u_{\Omega}|)^2x.\nu.$$
Since $|\nabla u_{\Omega}|)\geq r$ on $\partial\Omega$, then
\begin{enumerate}
     \item either $\mathcal{QS}(u_\Omega,r^2|\Omega|/(N+2))$
     \item or $\Omega$ is a ball ($|\nabla u_{\Omega}|=r$ on $\partial\Omega$)
 \end{enumerate}
 Observe that according to the same identity, $\mathcal{QS}(u_\Omega,r^2(x.\nu)/(N+2))$ has a solution.
%%%%%%%%%%%%%%%%%%%%%%%%%%%%%%%%%%%%%%%%%%%%%%%%%%%%%%%%%%%%%%%%%%%%%%%%%%%%%%%%%%%%%%%%%%%%%%%%%%%%%%%%%%%%%%%%
%%%%%%%%%%%%%%%%%%%%%%%%%%%%%%%%%%%%%%%%%%%%%%%%%%%%%%%%%%%%%%%%%%%%%%%%%%%%%%%%%%%%%%%%%%%%%%%%%%%%%%%%%%%%%%%%%%%

Mohammed Barkatou\\
ISTM, Departement of Mathematics\\
barkatou.m@ucd.ac.ma\\
Chouaib Doukkali University, Morocco\\


\begin{thebibliography}{2}
\bibitem{ac}  H. W. Alt and L. A. Caffarelli: Existence and regularity for a minimum problem with free boundary,
\emph{J. reine angew. Math.}, \textbf{325}, 1981, 434--448

\bibitem{ba1}  M. Barkatou: Some geometric properties for a class
of non Lipschitz-domains, \emph{New York J. Math.}, \textbf{8},
2002, 189--213.

\bibitem{ba2}  M. Barkatou: Necessary and Suffisant Condition of Existence
for the Quadrature Surfaces Free Boundary Problem , \emph{JMR}, \textbf{2} (4),
2010, 93--99.

\bibitem{ba3}  M. Barkatou: Existence and symmetry results for some overdetermined free boundary problems , \emph{Applied Sciences}, \text{Vol. 22}, 2020, pp. 33--51.

\bibitem{bsl}  M. Barkatou, D. Seck and I. Ly: An
existence result for a quadrature surface free boundary problem,
\emph{Cent. Eur. Jou. Math.}, \textbf{3}(1), 2005, 39--57.

\bibitem{bk1}  M. Barkatou and S. Khatmi: Symmetry result for some overdetermined value problems,
\emph{ANZIAM J.}, \textbf{49}, 2008, 479--494.

\bibitem{bk}  M. Barkatou and S. Khatmi: Existence of quadrature surfaces for
uniform density supported by a segment,
\emph{Applied Sciences}, \textbf{10}, 2008, 38--47.

\bibitem{be}  A. Beurling: On free boundary problem for the Laplace equation,
\emph{Sem. Anal. Funct., Inst. Adv. Study Princeton}, \textbf{1},
1957, 248--263

\bibitem{bh}  F. Brock and A. Henrot: A symmetry result for an overdetermined elliptic problem
using continuous rearrangement and domain derivative,
\emph{Rend. Circ. Mat. Palermo}, \textbf{51}, 2002, 375--390.

\bibitem{bal}  D. Bucur, V. Ferone, C. Nitsch and C. Trombetti: Weinstock inequality in higher dimensions
\emph{J. Differential Geom.}, \textbf{118} (1), May 2021, 1--21. https://doi.org/10.4310/jdg/1620272940

\bibitem{bz}  D. Bucur and J. P. Zolesio: N-dimensional shape optimization under capacitary constraints,
\emph{J. Diff. Eq.}, \textbf{123-2}, 1995, 504--522.

\bibitem{bt}  D. Bucur and P. Trebeschi: Shape Optimization
Problems Governed by Nonlinear State Equations, \textsl{Proc.Roy.Sc.Edinburgh%
}, vol. \textbf{128A}, 1998, pp. 945-963
\bibitem{ca}  T. Carleman: \"{U}ber ein Minimumproblem der mathematischen Physik,
\emph{Math. Z.}, \textbf{1}, 1918, 208--212

\bibitem{hc}  M. Choulli and A. Henrot: Use of domain derivative to prove symmetry results
in partial differential equations,
\emph{Math. Nachr.}, \textbf{192}, 1998, 91--103.

\bibitem{cm}  G. Ciraolo and F. Maggi: On the Sharpe of Compact Hypersurfaces with Almost-Constant Mean Curvature,
\emph{Comm. Pure. App. Mat.}, \textbf{0001-0052}, 2016, Preprint.

\bibitem{d}  J. Dalphin. Uniform ball condition and existence of optimal shapes for geometric functionals involving boundary-value problems \emph{2017 hal-01456344v2.}

\bibitem{fr}  K. Friedrichs: \"{U}ber ein Minimumproblem f\"{u}r Potentialstr\"{o}mungen mit freiem Rand,
\emph{Math. Ann.}, \textbf{109}, 1934, 60--82

\bibitem{fm}  S. J. Fromm and P. Mcdonald: A symmetry problem
from probability. \emph{ Proc. Amer. Math. Soc.}, \textbf{125},
1997, 3293--3297.

\bibitem{gwn}  G. Gidas, Wei-Ming Ni and L. Nirenberg: Symmetry
and related properties via the maximum principle. \emph{Comm.
Math. Phys.}, \textbf{68}, 1979, 209--300.

\bibitem{gs} B. Gustafsson and H. Shahgholian: Existence and geometric
properties of sokutions of a free boundary problem in potential theory,
\emph{J. f\"{u}r die Reine und Ang. Math.}, \textbf{473}, 1996, 137--179.

\bibitem{he} A. Henrot: Subsolutions and supersolutions in a free boundary problem, \emph{Arkiv f\"{o}r
Math.}, Vol. 32(1), 1994, pp. 79--98.

\bibitem{hp}  A. Henrot and M. Pierre: Variation et Optimisation de forme,une analyse géométrique, \emph{Mathématiques et Applications},
2005, Vol. \text\bf{48} Springer.

\bibitem{hm} C. Huang and D. Miller: Domain functionals and exit times for Brownian motion.
\emph{\ Proc. Amer. Math. Soc.}, \textbf{130}(3), 2001, 825--831.

\bibitem{kb} S. Khatmi and M. Barkatou: On some overdetermined free boundary problems.
\emph{\ ANZIAM J.}, \textbf{49}(E), 2007, E11--E32.


\bibitem{mp}  R. Magnanini and G. Poggesi: On the stability for Alexandrov's Soap Bubble theorem, \emph{JAMA.},
\textbf{139}, 2019, 179--205.

\bibitem{ps}  L. E. Payne and P. W. Schaefer: On overdetermined boundary value problems for the biharmonic operator, \emph{J. Math. Anal. and Appl.}, \textbf{187}, 1994, 598--616.

\bibitem{pi}  O. Pironneau: Optimal shape design for
elliptic systems, \emph{Springer Series in Computational Physics},
1984, Springer, New York.

\bibitem{Re}  F. Rellich: Darstellung der Eigenwerte
$\triangle u+\lambda u$ durch einem Randintegral , \emph{Math.
Z.}, \textbf{46}, 1940, 635--646.

\bibitem{se}  J. Serrin: A symmetry problem in potential theory, \emph{Arch. Rational Mech. Anal.},
\textbf{43}, 1971, 304--318.

\bibitem{si}  V. G. Sigillito: Explicit a priori inequalities with applications to boundary value problems, \emph{Pitman Publishing},
1977, London, San Francisco, Melbourne.

\bibitem{sz} J. Sokolowski and J. P. Zolesio: Introduction
to shape optimization: shape sensitivity analysis, \emph{Springer
Series in Computational Mathematics}, \textbf{10}, 1992, Springer,
Berlin.

\bibitem{val} M. Van den Berg, G. Buttazzo and B. Velichkov: Optimization Problems Involving the First Dirichlet Eigenvalue and the Torsional Rigidity, \emph{New Trends in Shape Optimization}, \textbf{166}, 2015, Aldo Pratelli G\"{u}nter Leugering Editors, Birkh\"{a}user,

\end{thebibliography}
\end{document}